\documentclass{article}
\usepackage{amsmath,amsthm,amssymb,amscd}

\newcommand{\ga}{\alpha}
\newcommand{\gb}{\beta}

\newcommand{\gd}{\delta}
\newcommand{\gw}{\omega}

\newcommand{\gS}{\Sigma}
\newcommand{\gs}{\sigma}

\newcommand{\coll}{\mathrm{Coll}}

\newcommand{\cantor}{2^\gw}

\newcommand{\dotxgen}{{\dot x}_{\mathit{gen}}}

\newcommand{\supp}{\mathrm{supp}}

\newcommand{\dom}{\mathrm{dom}}

\newcommand{\rng}{\mathrm{rng}}
\newcommand{\power}{\mathcal{P}}

\newtheorem{theorem}{Theorem}[section]

\newtheorem{claim}[theorem]{Claim}
\newtheorem{corollary}[theorem]{Corollary}

\newtheorem{proposition}[theorem]{Proposition}

\theoremstyle{definition}
\newtheorem{definition}[theorem]{Definition}
\newtheorem{example}[theorem]{Example}
\newtheorem{question}[theorem]{Question}

\title{Krull dimension in set theory\footnote{2010 AMS subject classification 03E15, 03E25, 03E35.}}

\author{
Jind{\v r}ich Zapletal\\
University of Florida}

\begin{document}
\maketitle

\begin{abstract}
For every number $n\geq 2$, let $\Gamma_n$ be the hypergraph on $\mathbb{R}^n$ of arity four consisting of all non-degenerate Euclidean rectangles. It is consistent with ZF+DC set theory that the chromatic number of $\Gamma_n$ is countable while that of $\Gamma_{n+1}$ is not.
\end{abstract}

\section{Introduction}

In culmination of several years of work, Schmerl~\cite{schmerl:avoidable} classified algebraic hypergraphs on Euclidean spaces according to how difficult it is to find countable coloring for them. It turns out that for every such a hypergraph $\Gamma$ there is a number $n\in\gw+1$ such that in ZFC, the assertion ``chromatic number of $\Gamma$ is countable" is equivalent to $2^{\aleph_0}\leq\aleph_n$ (if $n\in\gw$), or it is outright provable (if $n=\gw$). Moreover, this number $n=n(\Gamma)$ can be found algorithmically. 

Schmerl's success raises a question: is it possible to find a similar classification of algebraic hypergraphs with the base theory being the choiceless ZF+DC set theory? After all, coloring an algebraic hypergraph of any complexity requires copious amounts of the Axiom of Choice. 

In this paper, I make tangible progress towards such a classification. Using the methods of geometric set theory \cite{z:geometric}, I isolate a general approach for separating chromatic numbers of algebraic hypergraphs which differ by dimension, even when their dimension is high. It appears to be difficult to state a satisfactory general theorem, so I treat a remarkable special case. Given a number $n\geq 2$, let $\Gamma_n$ be the hypergraph on $\mathbb{R}^n$ of arity four consisting of all non-degenerate Euclidean rectangles. By a result of \cite{komjath:three}, in ZFC the chromatic number of every single of these hypergraphs being countable is a statement equivalent to the Continuum Hypothesis. In the choiceless theory ZF+DC, a more informative picture appears. 

\begin{theorem}
\label{maintheorem}
Let $n\geq 2$ be a number. It is consistent relative to an inaccessible cardinal that ZF+DC holds and the chromatic number of $\Gamma_n$ is countable while that of $\Gamma_{n+1}$ is not.
\end{theorem}

\noindent Unlike many purely combinatorial independence proofs in ZFC, the proof of Theorem~\ref{maintheorem} must consider algebraic structure of the real line, in particular the Hilbert basis theorem and the Krull dimension of the topology of algebraic subsets of $\mathbb{R}^n$. Many similar theorems can be proved on what at the moment is a case by case basis. The method introduced in this paper is not universally applicable. For example, the following question appears to require a different approach.

\begin{question}
Let $n\geq 2$ be a number. Writing $\Delta_n$ for the hypergraph of arity four of parallelograms on $\mathbb{R}^n$, is it consistent with ZF+DC that the chromatic number of $\Delta_n$ is countable while that of $\Delta_{n+1}$ is not?
\end{question}

\noindent The paper uses the set theoretic notational standard of \cite{jech:newset}. The calculus of balanced forcing is developed in \cite{z:geometric}. For purposes of this paper, a rectangle in a Euclidean space is always identified with the set of its four extreme points. A non-degenerate rectangle is one which has four extreme points, as opposed to degenerate rectangles which may only have one or two extreme points. Given a set $X$, a hypergraph $\Gamma$ on $X$ is just a collection of finite subsets of $X$, its elements are called hyperedges. A $\Gamma$-coloring is a (possibly partial) map $c$ on $X$ such that there is no $\Gamma$-hyperedge which is a subset of $\dom(c)$ and on which $\Gamma$ is constant. The chromatic number of $\Gamma$ is countable if there is a total $\Gamma$-coloring whose range is a subset of $\gw$. DC denotes the Axiom of Dependent Choices.

\section{Krull dimension}

\noindent In \cite{z:noetherian}, the notion of Noetherian topology is analyzed in descriptive set theoretic context. The following definition is central.

\begin{definition}
Let $X$ be a Polish $K_\gs$ space and let $\mathcal{T}$ be a topology on $X$ different from the original Polish one. Say that $\mathcal{T}$ is an \emph{analytic Noetherian topology} if

\begin{enumerate}
\item $\mathcal{T}$ is Noetherian. That is, there is no infinite strictly decreasing sequence of $\mathcal{T}$-closed sets;
\item $\mathcal{T}$ is analytic. That is, every $\mathcal{T}$-closed set is closed in the Polish topology, and the collection of $\mathcal{T}$-closed sets is analytic in the Effros Borel space on $X$.
\end{enumerate}
\end{definition}

\noindent The restriction to $K_\gs$ spaces serves to establish the Borelness of the union and intersection operations and the subset relation on closed sets as in \cite[Section 12.C]{kechris:classical}, accommodating the necessary complexity computations and absoluteness arguments.  
In this paper, I show that it is useful to consider the usual notion of dimension of Noetherian spaces, for use in certain choiceless independence results. Recall:

\begin{definition}
Let $X$ be a set and $\mathcal{T}$ be a Noetherian topology on it. A $\mathcal{T}$-closed set $C\subset X$ is \emph{irreducible} if it is not the union of a finite collection of $\mathcal{T}$-closed proper subsets.
\end{definition}

\noindent It is well-known and easy to prove that every $\mathcal{T}$-closed set $C$ can be written in a unique way as a union of a finite collection of nonempty irreducible sets, so-called irreducible components of $C$. For an analytic Noetherian topology $\mathcal{T}$ and its closed set $C\subset X$, its irreducibility is a coanalytic statement, therefore absolute between transitive models of set theory.

\begin{definition}
Let $X$ be a set and $\mathcal{T}$ be a Noetherian topology on it. The \emph{Krull dimension} of $\mathcal{T}$ is the supremum of the lengths of finite chains of nonempty irreducible $\mathcal{T}$-closed sets linearly ordered by strict inclusion, minus one.
\end{definition}

\noindent This standard definition may only yield values of dimension between $0$ and $\gw$ inclusive. It is not difficult to see that if $\mathcal{T}$ is an analytic Noetherian topology on a $K_\gs$ Polish space $X$, then the set of reducible $\mathcal{T}$-closed sets is analytic. The statement ``the Krull dimension of $\mathcal{T}$ is $n$'' is a conjunction of a $\mathbf{\Pi}^1_2$ and a $\mathbf{\gS}^1_2$-statements and as such absolute among all forcing extensions. While there are many interesting examples of Noetherian spaces with prescribed dimensions, this paper is limited to the following:

\begin{example}
\label{example1}
Let $n\geq 1$ and let $\mathcal{T}$ be the topology of algebraic sets on $\mathbb{R}^n$. Then $\mathcal{T}$ is analytic Noetherian and has Krull dimension $n$. This is the main corollary of the Hilbert Basis Theorem.
\end{example}

\noindent The following minor variation comes very handy in this paper.

\begin{example}
\label{example2}
Let $n\geq 1$ and let $\mathcal{T}$ be the topology on $\mathbb{R}^n\times\mathbb{R}$ whose closed sets are (the whole space and) finite unions of sets of the form $f\restriction A$, where $A\subset\mathbb{R}^n$ is an algebraic set distinct from $\mathbb{R}^n$, and $f$ is a polynomial function from $\mathbb{R}^n$ to $\mathbb{R}$. Then $\mathcal{T}$ is Noetherian of Krull dimension $n$. To see this, note that the set of generators is closed under intersections: $(f_0\restriction A_0)\cap (f_1\restriction A_1)=f_0\restriction (A\cap B\cap \{x\in\mathbb{R}^n\colon f_0(x)= f_1(x)\})$ and the last set in the intersection is algebraic. Every decreasing sequence of $\mathcal{T}$-indecomposable sets starts possibly with $\mathbb{R}^n\times\mathbb{R}$ and continuous with $f\restriction A_i$ for some fixed polynomial function $f\colon\mathbb{R}^n\to\mathbb{R}$ and indecomposable algebraic sets $A_i\subset\mathbb{R}^n$ distinct from the whole space. As such, it can have length at most $n+1$, and the Krull dimension of $\mathcal{T}$ is $n$. Note that if polynomial functions on the whole space $\mathbb{R}^n$ were allowed as generators of $\mathcal{T}$, the Krull dimension of the topology would increase to $n+1$.
\end{example}

\noindent The main feature connecting Krull dimension with forcing is the following mysterious theorem.

\begin{theorem}
\label{neatproposition}
Let $n\geq 2$ be a number and $\mathcal{T}$ be an analytic Noetherian topology on a $K_\gs$-Polish space $X$ of Krull dimension less than $n$. Let $\langle V[G_i]\colon i\in n\rangle$ be a tuple of forcing extensions such that for any sets $b_0, b_1\subset n$, $V[G_i\colon i\in b_0]\cap V[G_i\colon i\in b_1]=V[G_i\colon i\in b_0\cap b_1]$. For every point $x\in X$ there is $i\in n$ such that the smallest $\mathcal{T}$-closed set coded in $V[G_i]$ containing $x$ is coded in $V$.
\end{theorem}

\noindent The intersection condition on the tuples of generic extensions is satisfied for example for a mutually generic tuple by the product forcing theorem. The point $x\in X$ does not have to belong to any of the models mentioned. Note that if $M$ is any generic extension and $x\in X$ is any point outside of $M$, the Noetherian property of the topology $\mathcal{T}$ shows that there indeed exists an inclusion-smallest $\mathcal{T}$-closed set coded in $M$ which contains $x$ as an element. In addition, this set must be irreducible since all of its composants are coded in $M$, and one of them must contain the point $x$.

\begin{proof}
Suppose that the models $V[G_i]$ for $i\in n$ and the point $x$ are given. For each set $b\subseteq n$ write $M_b=V[G_i\colon i\in b]$ and let $C_b\subset X$ be the smallest $\mathcal{T}$-closed set coded in $M_b$ such that $x\in C_b$. Note that $c\subseteq b$ implies $C_b\subseteq C_c$; nevertheless, the status of sharp inclusion is not immediately clear. Suppose towards a contradiction that the conclusion of the proposition fails, i.e.\ $C_0\neq C_{\{i\}}$ for any $i\in n$.

For a set $b\subseteq n$, a \emph{good sequence in $b$} is an inclusion-increasing sequence $\langle b_k\colon k\leq |b|\rangle$ of subsets of $b$ such that $|b_k|=k$ and such that the sets $C_{b_k}$ strictly decrease with $k$. By induction on the cardinality of the set $b\subseteq n$ I will show that there is a good sequence in $b$. The base case $|b|=1$ follows from the initial contradictory assumption. For the induction step, suppose that $b\subset n$ is a set of cardinality at least two such that the statement has been verified for all its proper subsets. Choose $i\in b$ and let $c=b\setminus \{i\}$. By the induction hypothesis applied to the set $c$, there must be $j\in c$ such that, writing $d=c\setminus \{j\}$, $C_d\neq C_c$ holds. Let $e=d\cup\{i\}$, so $d=e\cap c$.

It cannot be the case that $C_e=C_c$. To see that, note that in such a case, $C_c$ would be coded in both $M_e$ and $M_c$, and by the intersection condition on the generic extensions, it would be coded in the model $M_d$, contradicting the assumption that $C_d\neq C_c$. Thus, the set $C_e\cap C_c$ is a proper subset either of $C_c$ or of $C_e$. If the former case prevails, use the induction hypothesis on $c$ to find a good sequence in $c$ and add to it $C_b$, which is a subset of $C_c\cap C_e$ and therefore a proper subset of $C_c$. Thus, a good sequence in $b$ has been obtained. If the latter case prevails, just switch the role of $c$ and $e$ and construct a good sequence in $b$ just the same.

After the induction is complete, consider a good sequence $\langle b_k\colon k\leq n\rangle$ in the set $b=n$. The sets $C_{b_k}$ for $k\leq n$ form a strictly decreasing sequence of length $n+1$ consisting of irreducible sets, contradicting the assumption on the Krull dimension of the topology $\mathcal{T}$.
\end{proof}

\section{Pairs of generic extensions}

The paper \cite{z:noetherian} introduces the notion of mutually Noetherian pair of generic extensions. In this section, I develop a stratification of that concept which takes into account the Krull dimension of the topologies concerned.

\begin{definition}
Let $V[G_0]$, $V[G_1]$ be generic extensions. The \emph{Krull dimension of $V[G_0]$ over $V[G_1]$}, $\dim(V[G_0]/V[G_1])$, is

\begin{enumerate}
\item at least $1$ always;
\item at least $n$ (where $n\geq 2$ is a natural number) if for every $K_\gs$ Polish space $X$, every basic open set $O\subset X$, and every analytic Noetherian topology $\mathcal{T}$ on $X$ of Krull dimension less than $n$, both in the ground model, if $C$ is a $\mathcal{T}$-closed set coded in $V[G_0]$ with nonempty intersection with $O\cap V[G_1]$, then $C$ has nonempty intersection with $O\cap V$;
\item equal to $n$ if $\dim(V[G_0]/V[G_1])\geq n$ and $\dim(V[G_0]/V[G_1])\not\geq n+1$;
\item $\infty$ if $\dim(V[G_0]/V[G_1])\geq n$ holds for every natural number $n$.
\end{enumerate}
\end{definition}

\noindent Noetherian topologies of dimension $0$ are uninteresting, and that is why the counting starts at $n=1$. The case $n=2$ allows a simple characterization via the following proposition. Higher dimensions are more difficult to understand, and they are the bread and butter of this paper. I do not know if $\dim(V[G_0]/V[G_1])=\dim(V[G_1]/V[G_0])$ holds in general, but this interesting question is irrelevant for the purposes of this paper.

\begin{proposition}
\label{twoproposition}
Let $V[G_0]$, $V[G_1]$ be generic extensions. $\dim(V[G_0]/V[G_1])\geq 2$ holds if and only if $\power(\gw)\cap V[G_0]\cap V[G_1]=\power(\gw)\cap V$.
\end{proposition}

\begin{proof}
For the left-to-right inclusion, suppose that $\dim(V[G_0]/V[G_1])\geq 2$ holds. Consider the topology $\mathcal{T}$ on $\power(\gw)$ whose closed sets are the whole space, the empty set, and the finite sets. It is an analytic Noetherian topology of Krull dimension $1$. Suppose that $x\in\power(\gw)$ belongs to $V[G_0]\cap V[G_1]$. Applying the dimension assumption to the set $C=\{x\}$, the conclusion is that $x\in V$ as desired. 

For the right-to-left inclusion, assume that $\power(\gw)\cap V[G_0]\cap V[G_1]=\power(\gw)\cap V$ holds. Let $X$ be a $K_\gs$ Polish space with a fixed countable basis, let $O\subset X$ be an open set, and let $\mathcal{T}$ be an analytic Noetherian topology of Krull dimension $1$, all coded in $V$. Let $C\in\mathcal{T}$ be a closed set coded in $V[G_0]$ and $x\in V[G_1]$ be a point in $C\cap O$. Consider the sets $D, D_0, D_1\in\mathcal{T}$ which are the smallest $\mathcal{T}$-closed sets containing $x$ and coded in $V, V[G_0], V[G_1]$ respectively. These sets do exist and are indecomposable as the topology $\mathcal{T}$ is Noetherian. Clearly, $D_0\subseteq D$ holds. In addition, a Shoenfield absoluteness argument shows that $D_1$ is the smallest $\mathcal{T}$-closed set containing $x$ in any model, so $D_1\subseteq D_0$ holds. 

The Krull dimension assumption on the topology $\mathcal{T}$ now implies that either $D_0=D$ or $D_1=D_0$ must hold. If $D_0=D$, then $D_0$ is coded in the ground model. By a Shoenfield absoluteness argument $D_0\cap O$ must be nonempty in $V$ since it is nonempty in $V[G_1]$. Any ground model point of $D_0\cap O$ verifies the $\dim(V[G_0]/V[G_1])\geq 2$ assertion. If, on the other hand, $D_0=D_1$ holds, then the set of all basic open sets disjoint from $D_0$ belongs to both $V[G_0]$ and $V[G_1]$, and therefore to $V$ by the initial assumption. Thus, $D_0$ is coded in $V$, and the proof is concluded as in the previous case.
\end{proof}

\noindent The next order of business is to show that certain common operations on generic extensions preserve Krull dimension.

\begin{proposition}
\label{productproposition}
Let $V[G_0]$ and $V[G_1]$ be generic extensions and $n\geq 1$ be a number. Suppose that $\dim(V[G_0]/V[G_1])\geq n$ holds.

\begin{enumerate}
\item If $V[K_0]\subset V[G_0]$ and $V[K_1]\subset V[G_1]$ then $\dim(V[K_0]/V[K_1])\geq n$;
\item if $P_0\in V[G_0]$ and $P_1\in V[G_1]$ are posets and $H_0\subset P_0$ and $H_1\subset P_1$ are filters mutually generic over $V[G_1, G_0]$, then $\dim(V[G_0][H_0]/V[G_1][H_1])\geq n$.
\end{enumerate}
\end{proposition}

\noindent In particular, a mutually generic pair of extensions of $V$ has Krull dimension $\infty$.

\begin{proof}
The first item is obvious. For the second item, let $X$ be a $K_\gs$-Polish space and $\mathcal{T}$ an analytic Noetherian topology on it of Krull dimension smaller than $n$. Suppose that $\tau_1\in V[G_1]$ is a $P_1$-name for a $\mathcal{T}$-closed subset of $X$ and $\tau_0\in V[G_0]$ is a $P_0$-name for an element of $X$, and $p_1\in P_1$ and $p_0\in P_0$ are conditions such that the pair $\langle p_1, p_0\rangle$ forces in the product $P_1\times P_0$ that $\tau_0\in\tau_1$ holds. I will find a point $x\in X\cap V$ such that $p_1\Vdash\check x\in \tau_1$, proving the proposition.

Working in $V[G_1]$, let $M_1$ be a countable elementary submodel of a large structure, and let $\mathcal{F}$ be the set of all filters on $P_1\cap M_1$ which are generic over the model $M_1$ and contain the condition $p_1$. Let $C=\bigcap\{\tau_1/F\colon F\in\mathcal{F}\}$; this is a $\mathcal{T}$-closed set in the model $V[G_1]$. Working in $V[G_0]$, let $M_0$ be a countable elementary submodel of a large structure, let $F_0\subset P_0$ be a filter generic over $M_0$ containing the condition $p_0$, and let $x_0=\tau_0/F_0$; this is a point in the model $V[G_0]$. Observe that $x_0\in C$ must hold. Otherwise, there would be a filter $F_1\in\mathcal{F}$ such that $x_0=\tau_0/F_0\notin\tau_1/F_1$ and a basic open set $O\subset X$ in its usual Polish topology which contains $\tau_0/F_0$ and is disjoint from $\tau_1/F_1$. By the forcing theorem applied with $P_1$ and $P_0$, there would have to be conditions $p'_1\leq p_1$ and $p'_0\leq p_0$ forcing respectively that $O\cap\tau_1=0$ and $\tau_0\in O$. Such conditions contradict the initial assumptions on $p_1$ and $p_0$.

Now, since $\dim(V[G_0]/V[G_1])\geq n$, there must be a point $x\in X\cap V$ such that $x\in C$ holds. I claim that $p_1\Vdash\check x\in\tau_1$ as desired. If this failed, there would have to be a basic open set $O\subset X$ containing the point $x$ and a condition $p'_1\leq p_1$ forcing $O\cap\tau_1=0$. By the elementarity of the model $M_1$, such a condition $p'_1$ can be found in the model $M_1$. Let $F\in\mathcal{F}$ be a filter containing the condition $p'_1$. Observe that $x\notin\tau_1/F$, contradicting the choice of the point $x$.
\end{proof}

\noindent The Krull dimension of generic extensions will be used in this paper only to topologies associated with algebraic subsets of Euclidean spaces. The following proposition provides their instrumental property reminiscent of mutual genericity.

\begin{proposition}
\label{iproposition}
Suppose that $n\geq 1$ is a number and $V[G_0], V[G_1]$ are generic extensions such that $\dim(V[G_0]/V[G_1])\geq n+1$. Suppose that $f\colon \mathbb{R}^n\to\mathbb{R}$ is a polynomial function and $C\subset\mathbb{R}^n$ is an algebraic set distinct from $\mathbb{R}^n$, both coded in $V[G_0]$. Suppose $x\in C$ is a point in $V[G_1]$ such that $f(x)\in V[G_0]$.
Then there is an algebraic set $D\subseteq C$ and a polynomial function $g\colon\mathbb{R}^n\to\mathbb{R}$, both coded in the ground model $V$, such that $x\in D$ and $f\restriction D=g\restriction D$.
\end{proposition}

\begin{proof}
Consider the analytic Noetherian topology $\mathcal{T}$ on $\mathbb{R}^n\times\mathbb{R}$ generated by all restrictions of a polynomial functions from $\mathbb{R}^n$ to $\mathbb{R}$ to an algebraic subset of $\mathbb{R}^n$ distinct from $\mathbb{R}^n$. As in Example~\ref{example2}, the Krull dimension of $\mathcal{T}$ is clearly equal to $n$. Note that $f\restriction C$ is a $\mathcal{T}$-closed set in $V[G_0]$ containing the pair $\langle x, f(x)\rangle$. By an application of the dimension assumption, there is a sequence of pairs $\langle x_i, y_i\colon i\in\gw\rangle$ converging to $\langle x, f(x)\rangle$ such that for each $i\in\gw$, the points $x_i\in C$ and $y_i$ are both in $V$ and $f(x_i)=y_i$.

Let $e$ be a polynomial function on $\mathbb{R}^n$ such that $C=\{z\in\mathbb{R}^n\colon e(z)=0\}$. Now, the polynomial functions $f, e$ are linear functions of the coefficients used in the polynomials defining $f, e$. That is, for some $k\in\gw$ there is a polynomial function $h\colon \mathbb{R}^m\times\mathbb{R}^n\to\mathbb{R}\times\mathbb{R}$ with integer coefficients which is linear in the variables from $\mathbb{R}^m$ and such that for some $t\in\mathbb{R}^m\cap V[G_0]$, for all $z\in\mathbb{R}^n$ $\langle f(z), e(z)\rangle=h(t, z)$. For each $i\in\gw$ let $F_i=\{u\in\mathbb{R}^m\colon h(u, x_i)=\langle y_i, 0\rangle\}$; this is an algebraic and even affine subset of $\mathbb{R}^m$. Since the topology of algebraic subsets of $\mathbb{R}^m$ is Noetherian, the intersection $F=\bigcap_i F_i$ belongs to the ground model, since for some finite number $j\in\gw$, $\bigcap_{i\in\gw}F_i=\bigcap_{i\in j}F_i$ holds. Note that the set $F$ is nonempty, because in $V[G_0]$ it contains $t$. Let $D\subset\mathbb{R}^n$ be the algebraic set of all points $z\in\mathbb{R}^n$ such that the value of $h(u, z)$ is the same for all $u\in F$ and has zero as its second coordinate. Note that $D$ is in the ground model,  $x_i\in D$ for all $i\in D$ and also $x\in D$, since all polynomial functions are continuous.
Choose a point $u\in F$ in the ground model, and let $g$ be the polynomial function defined by $g(z)=$ the first coordinate of $h(u, z)$. Clearly, the set $D$ and the function $g$ work as required.
\end{proof}

\section{Examples}

In order to use the concept of Krull dimension efficiently, it is necessary to build examples of pairs of generic extensions with interesting Krull dimension characteristics. In this section, I show how to produce such examples from algebraic subsets of Euclidean spaces. The following abstract observation will be useful at various points in the proofs.

\begin{theorem}
\label{governtheorem}
Let $X_0, X_1$ be Euclidean spaces and $A\subset X_0\times X_1$ be an algebraic set, and write $\pi\colon X_0\times X_1\to X_0$ for the projection function. For every nonempty relatively open set $O_0\subset A$ there is a nonempty relatively open set $O_1\subset O_0$ and  an algebraic set $B\subset X_0$ such that $\pi\restriction O_1$ is an open map to $B$. 
\end{theorem}

\begin{proof}
Fix the relatively open set $O_0\subset A$. To find the algebraic set $B$, consider the collection $\mathcal{C}$ of all algebraic sets $C\subset X_0$ such that there is a nonempty basic relatively open set $O'\subset O_0$ such that $\pi''O'\subset C$. Note that $\mathcal{C}$ is nonempty, containing the whole space $X_0$ in particular. By the Hilbert basis theorem, there is an inclusion-minimal element $B\in\mathcal{C}$, with its membership witnessed by some basic relatively open set $O'\subset O_0$. Below, we need the following key feature of the set $B$.

\begin{claim}
For every algebraic set $D\subset X_1$, either $B\subseteq D$ holds or $\pi^{-1}D$ is nowhere dense in $O'$. 
\end{claim}

\begin{proof}
Observe that if $\pi^{-1}D$ is somewhere dense in $O'$, then since $\pi^{-1}D$ is a closed set, it has to contain a nonempty relatively open subset $O''\subseteq O'$. By the minimal choice of $B$, it must be the case that $D\cap B=B$ or in other words $B\subseteq D$.
\end{proof}

\noindent To find the set $O_1$, let $C$ be the set of those $x\in O'$ such that $f\restriction O'$ is an open map to $B$ at $x$. In other words, $C=\{x\in O'\colon$ for every $U\subset O'$ basic open with $x\in U$ there is $V\subset X_0$ basic open containing $\pi(x)$ such that $V\cap B\subset \pi'' U\}$. Note that $C\subset X_0\times X_1$ is a semialgebraic set. By the quantifier elimination theorem for real closed fields \cite[Theorem 3.3.15]{marker:book}, $C$ is a finite Boolean combination of open and closed sets, and therefore there is a nonempty relatively open set $O_1\subset O'$ such that either $O_1\subset C$ or $O_1\cap C=0$. The former option gives us the conclusion of the theorem. It is therefore enough to derive a contradiction from the latter option.

Assume towards a contradiction that $O_1\cap C=0$. Let $\mathcal{B}$ be a countable basis of semialgebraic relatively open subsets of $O_1$. For each $U\in\mathcal{B}$, consider the set $\pi''U\subset B$. This is a semialgebraic set; by the quantifier elimination theorem, it is a Boolean combination of some sets of the form $\{y\in X_0\colon p(y)=0\}$ and $\{y\in X_0\colon p(y)>0\}$ for some polynomials $p$ with real coefficients. Let $p_{Ui}$ for $i\in i_U$ be a list of these polynomials. For each $i\in i_U$, the algebraic set $D_{iU}=\{y\in X_0\colon p_{iU}(y)=0\}$ is either a superset of $B$, or else its $\pi$-preimage is nowhere dense in $O'$ by the claim. Use the Baire category theorem to find a point $x\in O_1$ which belongs to no set  $\pi^{-1}D_{iU}$ where $U\in\mathcal{B}$, $i\in i_{U}$, and $B\not\subseteq D_{iU}$.

Since the point $x$ does not belong to the set $C$, the map $\pi\restriction O'$ is not open to $B$ at $x$. Thus, there must be a basic relatively open set $U\in\mathcal{B}$ such that $x\in U$ and $\pi''U$ contains no open neighborhood of $f(x)$. By the choice of the point $x$, the point $f(x)$ belongs to none of the closed sets $D_{iU}\subset X_0$ except to those for which $B\subseteq D_{iU}$.  Let $V\subset X_0$ be an open neighborhood of $f(x)$ such that for every $i\in i_U$, if $B\not\subseteq D_{iU}$ then $V\cap D_{iU}=0$, and if $p_{iU}(f(x))\neq 0$ then $p_{iU}$ does not change sign in $V$.
The choice of the neighborhood $V$ shows that within the set $V\cap B$, the membership in the set $\pi''U$ does not change. Since $f(x)\in V\cap B$, it must be that $V\cap B\subseteq \pi''U$ holds. This contradicts the choice of the open set $U\subset A$.
\end{proof}

\noindent All generic extensions discussed in this section are Cohen generic with the following parlance. If $X$ is a Polish space then $P_X$ denotes the \emph{Cohen forcing} associated with $X$, the set of nonempty open subsets of $X$ ordered by inclusion, adding a single generic point $\dotxgen$. If $Y$ is another Polish space and $f\colon X\to Y$ is a continuous open map, then $P_X\Vdash\dot f(\dotxgen)$ is a Cohen generic point over $V$ for the poset $P_Y$ by \cite[Proposition 3.1.1]{z:geometric}. An important point is that if $n\geq 1$ is a number, then the posets $P_{X^n}$ and $(P_X)^n$ are co-dense, therefore a $P_{X^n}$-generic $n$-tuple consists of mutually $P_X$-generic points.

\begin{definition}
Suppose that $X, Y$ are Polish spaces, $f\colon X\to Y$ is a continuous function, and $n\geq 1$ is a number. $f^n\colon X^n\to Y^n$ denotes the continuous function defined by $f^n(\vec x)(i)=f(\vec x(i))$. $X^n/f$ is the closed subset of $X^n$ of all points $\vec x\in X^n$ such that the values $f(\vec x(i))$ are the same for all $i\in n$. In particular, $X^1/f=X$.
\end{definition}

\begin{proposition}
\label{triviproposition}
Let $X, Y$ be Polish spaces, $f\colon X\to Y$ be a continuous open function, and $a\subset b$ be nonempty finite sets. The projection function $\pi\colon X^b/f\to X^a/f$ is open.
\end{proposition}

\begin{proof}
Suppose that $O_i\subset X$ are open sets for $i\in b$, and for each $j\in a$ let $O'_j=O_j\cap \bigcap_{i\in b}f^{-1}f''O_i$. The sets $O'_j\subset O_j$ are open since the function $f$ is open. It is not difficult to see that $\pi''(\prod_{i\in b}O_i\cap X^b/f)=\prod_{j\in a}O'_j\cap X^a/f$.
\end{proof}

\noindent Now I am ready to state and prove a theorem which connects dimension of algebraic sets with Krull dimension of pairs of generic extensions. For algebraic and semi-algebraic subsets of Euclidean spaces I use the tame topological notion of dimension as developed in \cite{vandendries:tame}.

\begin{theorem}
\label{dimtheorem}
Let $X$ be an algebraic subset of a Euclidean space, and $f_0\colon X\to Y_0$ and $f_1\colon X\to Y_1$ be projections of $X$ to certain sets of coordinates. Let $n\geq 1$ be a number. If for every nonempty relatively basic  open set $U\subset X^n/f_1$, the set $(f_0^n)''U\subset Y_0^n$ has the maximal possible dimension $n\cdot\dim(Y_0)$, then $P_X\Vdash \dim(V[f_0(\dotxgen)]/V[f_1(\dotxgen)])\geq n$.
\end{theorem}

\begin{proof}
Let $O$ be a nonempty, relatively basic open subset of $X$, forcing the conclusion to fail. Thinning out $O$ if necessary, I may assume that $f_0$ and $f_1$ on $O$ are open maps to some algebraic subsets $A_0\subseteq Y_0$, $A_1\subseteq Y_1$ respectively (this is by Theorem~\ref{governtheorem}), and there is a Polish $K_\gs$-space $Z$, a nonempty open set $W\subset Z$, an analytic Noetherian topology $\mathcal{T}$ of Krull dimension less than $n$, and $P_{A_0}$-name $\tau_0$ for a $\mathcal{T}$-closed subset of $Z$ containing no ground model elements of $W$, and a $P_{A_1}$-name $\tau_1$ for an element of $U$ such that $O\Vdash\tau_1/f_1(\dotxgen)\in\tau_0/f_0(\dotxgen)$.

Consider the nonempty relatively open set $U_0=O^n\cap X^n/f_1$. Observe that $X^n/f_1$ is an algebraic set. By Theorem~\ref{governtheorem}, there is a relatively open nonempty subset $U_1\subset U_0$ on which $f_0^n$ is an open map to an algebraic set $A_2\subset Y_0^n$. Now, the algebraic set $A_2$ has full dimension $n\cdot \dim(Y_0)$ by the assumptions, so $A_2=Y_0^n$ must hold.
Let $\vec x\in U_1$ be a point generic over $V$ for the poset $P_{X^n/f_1}$. By Proposition~\ref{triviproposition}, each of the points $\vec x(i)\in X$ for $i\in n$ is $P_X$-generic over $V$, their $f_1$-images coincide and are equal to a point $y_1\in Y_1$ which is $P_{A_1}$-generic over $V$, and their $f_0$-images $y_{0i}=f_0(\vec x(i))$ are mutually generic elements of $Y_0$. Consider the models $V[y_{0i}]$ for $i\in n$, and in each of them the $\mathcal{T}$-closed set $C_i=\tau_0/y_{0i}$. Consider also the point $z=\tau_1/y_1$. By Theorem~\ref{neatproposition}, there is a number $i\in n$ and a ground model coded $\mathcal{T}$-closed set $D\subset Z$ such that $z\in D\subseteq C_i$. Now, the set $D$ contains a point in the open set $W$, namely $z$; by a Mostowski absoluteness argument, it must contain also some other point $z'\in W$ in the ground model. Then $z'\in C_i$ holds, contradicting the initial assumption about the name $\tau_0$.
\end{proof}

\noindent The dimension demand in the assumptions of Theorem~\ref{dimtheorem} may seem hard to verify; the following notion will prove helpful for the ends of this paper.

\begin{definition}
\label{almostdefinition}
Let $A\subset X_0\times X_1$ be an algebraic subset of the product of two Euclidean spaces. Say that $A$ is \emph{narrow} if

\begin{enumerate}
\item for every nonempty relatively open set $O\subset A$ with nonempty intersection with $A$, $\dim(O)\geq\dim(X_0)$;
\item there is a relatively open dense subset $C\subset A$ with all vertical sections finite.
\end{enumerate}
\end{definition}

\begin{proposition}
\label{dproposition}
Suppose $A\subset X_0\times X_1$ is a narrow algebraic subset of the product $X_0\times X_1$ of two Euclidean spaces. Let $\pi\colon A\to X_0$ be the projection function. For every nonempty relatively open set $O\subset A$, $\dim(\pi'' O)=\dim(X_0)$.
\end{proposition}

\begin{proof}
Fix the set $O\subset A$. By (2) of Definition~\ref{almostdefinition}, shrinking $O$ if necessary I may assume that $O$ has all vertical sections finite. The only way to account for (1) of Definition~\ref{almostdefinition} is to conclude that $\dim(\pi'' O)=\dim(X_0)$.
\end{proof}

\noindent For all of the following examples, fix a number $n\geq 2$.

\begin{example}
\label{7example}
Let $X_0=(\mathbb{R}^n)^n\times (\mathbb{R}^n)^n$, let $X_1=\mathbb{R}^n$, and $A=\{\langle x_i\colon i\in n, y_i\colon i\in n, z\rangle\in X_0\times X_1\colon$ for all $i\in n$, $(x_i-z)\cdot (y_i-z)=0\}$.  The set $A\subset X_0\times X_1$ is narrow.
\end{example}

\begin{proof}
The following claim will be useful. Let $d$ denote the Euclidean metric on $\mathbb{R}^n$.

\begin{claim}
\label{sclaim}
If $x_i$ for $i\in n$ are linearly independent points in $\mathbb{R}^n$ and $r_i\in\mathbb{R}$ are real numbers, then there are at most two points $z\in\mathbb{R}^n$ such that $d(x_i, z)=r_i$ holds for all $i\in n$.
\end{claim}

\begin{proof} 
Write the equations for such points $z$: the $i$-th equation is $z\cdot z-2x_i\cdot z+x_i\cdot x_i=r_i^2$. Subtract the first equation from the others to remove the quadratic term and get a system of $n-1$ many equations $2(x_0-x_i)\cdot z=r_i^2-r_0^2+x_0\cdot x_0-x_i\cdot x_i$ for $0<i<n$. Since the vectors of coefficients on the left hand side are linearly independent, the set of solutions is a line. This line intersects the hypersphere around $x_0$ of radius $r_0$ in at most two points, and these are the only points $z\in\mathbb{R}^n$ such that $d(x_i, z)=r_i$ holds for all $i\in n$.
\end{proof}

To check the items of Definition~\ref{almostdefinition}, suppose that $O\subset X_0\times X_1$ is an open set with nonempty intersection with $A$. Counting the dimensions of $A\cap O$ yields the following. The $z$ and $x$ coordinates can be chosen arbitrarily within certain open subsets of $\mathbb{R}^n$ ($n+n^2$ many dimensions) after which the $y$-coordinates must be chosen on the hyperplane which contains $z$ and is perpendicular to $z-y_i$ ($n(n-1)$ many dimensions). Together, this yields $n+n^2+n(n-1)=2n^2$ many dimensions as desired.

Consider the set $P=\{\langle x_i, y_i, z\colon i\in n\rangle\colon$ the center points of the segments connecting $x_i$ and $y_i$ form a linearly independent set$\}\subset X_0\times X_1$. The set $P\subset X_0\times X_1$ is open as the linear independence is checked by the nonzero determinant test. The vertical sections of $P\cap A$ are finite by Claim~\ref{sclaim} and Thalet's theorem. To check that $P\cap A$ is dense in $A$, let $O\subset X_0\times X_1$ be an open set with nonempty intersection with $A$. To find a point in $P\cap O\cap A$, choose $z$ arbitrarily in a certain open set, and choose the segment centers $w_i$ for $i\in n$ in certain open sets in a linearly independent way, using the fact that every nonempty open set spans the whole space $\mathbb{R}^n$. Then complete the construction by finding the points $x_i, y_i$ for $i\in n$. 
\end{proof}

\begin{example}
\label{8example}
Let $X_0=(\mathbb{R}^n)^n\times (\mathbb{R}^n)^n$, let $X_1=\mathbb{R}^n$, and $A=\{\langle x_i\colon i\in n, y_i\colon i\in n, z\rangle\in X_0\times X_1\colon$ for all $i\in n$, $(x_i-y_i)\cdot (x_i-z)=0\}$.  The set $A\subset X_0\times X_1$ is narrow. 
\end{example}

\begin{proof}
There is again a preliminary claim.

\begin{claim}
If $\langle x_i\colon i\in n, y_i\colon i\in n\rangle\in X_0$ is a point such that the differences $x_i-y_i$ for $i\in n$ are linearly independent, then the vertical section of $A$ above the point has exactly one element.
\end{claim}

\begin{proof}
The linear equations for the points in the vertical section have a matrix consisting of the vectors $x_i-y_i$ for $i\in n$.
\end{proof}

To verify the items of Definition~\ref{almostdefinition}, suppose that $O\subset X_0\times X_1$ is an open set with a nonempty intersection with $A$. To evaluate the dimension of $O\cap A$, first choose a point $z\in X_1$ in a certain open set ($n$ dimensions), then the points $y_i$ in certain open sets ($n^2$ many dimensions), and then points $x_i$ in the hyperplanes perpendicular to the segments $x_i-y_i$ and passing through $y_i$ ($n(n-1)$ dimensions). In total, it gives dimension $n+n^2+n(n-1)=n^2$ as desired. Now, consider the set $P\subset X_0\times X_1$ consisting of all points $\langle x_i, y_i\colon i\in n, z\rangle$ such that $x_i-y_i$ are linearly independent for $i\in n$. This is an open set by the nonzero determinant test for linear independence. The claim shows that $P\cap A$ has vertical sections of cardinality exactly one. To see that $P\cap A$ is dense in $A$, let $O\subset X_0\times X_1$ be an open set with nonempty intersection with $A$. To find a point in $O\cap P\cap A$, choose $z$ and $x_i$ for $i\in n$ in certain open neighborhoods, and then choose the points $y_i$ on the hyperspheres with diameters $[x_i, z]$, making sure that $x_i-y_i$ for $i\in n$ are linearly independent points. This is possible since every nonempty open subset of a hypersphere spans the whole space $\mathbb{R}^n$. 
\end{proof}

\begin{example}
\label{rexample}
Let $n\geq 2$ be a number and $X\subset (\mathbb{R}^n)^4$ be the closed set of all rectangles in $\mathbb{R}^n$. Let $x\in X$ be a point $P_{X}$-generic over $V$.

\begin{enumerate}
\item For every $i\in 4$, the point $x(i)\in\mathbb{R}^n$ is generic over $V$;
\item for every $i\neq j$ in $4$, $V[x(i)]$ and $V[x(j)]$ are mutually generic;
\item for any three pairwise distinct indices $i, j, k\in 4$, $\dim(V[x_i, x_j]/V[x_k])\geq n$.
\end{enumerate}
\end{example}

\begin{proof}
The first two items follow from the fact  that the projection from $X$ to any one or two coordinates is an open function to $\mathbb{R}^n$ or $(\mathbb{R}^n)^2$. This is easy and left to the reader. For the third item, first observe that the fourth point of a rectangle is a continuous function of the other three, and the projection from a graph of a continuous function is an open map. Thus, the projection from $X$ to any three coordinates is an open function to $Y$, the set of all right-angle triangles in $(\mathbb{R}^n)^3$. The dimension estimates for $P_Y$-generic triples now follow from Theorem~\ref{dimtheorem}, Proposition~\ref{dproposition}, and Examples~\ref{7example} and~\ref{8example}.
\end{proof}

\section{A preservation theorem}

The notion of Krull dimension of pairs of generic extension has a variation of balance associated to it, leading to a useful preservation theorem regarding certain balanced extensions of the choiceless Solovay model. This is explained in the present section.

\begin{definition}
A \emph{dimension characteristic} is a nonempty finite sequence whose entries are positive natural numbers or the $\infty$ symbol. Dimension characteristics of the same length are ordered by coordinatewise ordering. Let $t$ be a dimension characteristic. A tuple $\langle V[G_i]\colon i\in |t|+1\rangle$ of generic extensions has \emph{dimension characteristic $t$} if for every $j\in\dom(t)$, $t(j)$ is the largest number such that nonempty set $a\subset |t|+1$ of cardinality $j+1$ and every $k\in (|t|+1)\setminus a$, $\dim(V[G_i\colon i\in a]/V[G_k])\geq t(j)$.
\end{definition}

\begin{example}
\label{example3}
Let $n\geq 2$ be a number and let $Q_{rn}$ be the Cohen poset on the space of all rectangles in $\mathbb{R}^n$. Let $\langle x_i\colon i\in 4\rangle$ be a quadruple generic over $V$ for the poset $Q_{rn}$. Then each point $x_i$ is a Cohen-generic point of $\mathbb{R}^n$ and the quadruple $\langle V[x_i]\colon i\in 4\rangle$ of generic extensions has dimension characteristic $\langle \infty, n, 1\rangle$. This follows from Example~\ref{rexample}.
\end{example}

\begin{definition}
Let $P$ be a Suslin poset and $t$ be a dimension characteristic.

\begin{enumerate}
\item If $\bar p$ be a virtual condition in $P$, we say that $p$ is \emph{$t$-dimensionally balanced} if for every tuple $\langle V[G_i]\colon i\in |t|+1\rangle$ of generic extensions of dimension characteristic $\geq t$, every tuple $\langle p_i\colon i\in j\rangle$ of conditions such that $p_i\in V[G_i]$ and $p_i\leq \bar p$ has a common lower bound;
\item the poset $P$ is {$t$-dimensionally balanced} if for every condition $p\in P$ there is a $t$-dimensionally balanced virtual condition $\bar p\leq p$.
\end{enumerate}
\end{definition}

\begin{theorem}
\label{preservationtheorem}
Let $n\geq 2$ be a natural number. In every extension of the choiceless Solovay model which is cofinally $\langle n, n, 1\rangle$-dimensionally balanced, every nonmeager subset of $\mathbb{R}^n$ contains all vertices of a non-degenerate rectangle.
\end{theorem}

\begin{proof}
Let $\kappa$ be an inaccessible cardinal. Let $P$ be a Suslin forcing which is $\langle n, n, 1\rangle$-dimensionally balanced cofinally below $\kappa$.  Let $W$ be the associated choiceless Solovay model and work in $W$. Suppose that $p\in P$ is a condition and $\tau$ is a $P$-name for a nonmeager subset of $\mathbb{R}^n$. I must find a rectangle and a strengthening of the condition $p$ which forces the rectangle to be a subset of $\tau$. The name $\tau$ and the condition $p$ are both definable from some parameter $z\in\cantor$ and a parameter in the ground model $V$. Let $V[K]$ be an intermediate extension by a poset of cardinality smaller than $\kappa$ such that $z\in V[K]$ and $P$ is $\langle n, n, 1\rangle$-dimensionally balanced  $V[K]$. 

Work in $V[K]$. By the balance assumption, there must be a $\langle n, n, 1\rangle$-dimensionally balanced virtual condition $\bar p\leq p$. Write $Q$ for the Cohen forcing associated with $\mathbb{R}^n$, adding a generic point $\dotxgen\in\mathbb{R}^n$. By the forcing theorem, there must be a condition $q\in Q$, a poset $R$, and a $Q\times R$-name $\gs$ for a condition in $P$ stronger than $\bar p$ such that $q\Vdash_Q R\Vdash\coll(\gw, <\kappa)\Vdash \gs\Vdash_P\dotxgen\in\tau$. Otherwise, in the model $W$, the virtual condition $\bar p$ would force $\tau$ to be disjoint from the comeager set of points in $\mathbb{R}^n$ which are Cohen generic over the model $V[K]$, contradicting the initial assumption on the name $\tau$.

Let $Q_{rn}$ be the Cohen poset on the space of rectangles in $\mathbb{R}^n$. In the model $W$, find a rectangle $\langle x_i\colon i\in 4\rangle$ generic over $V[K]$ for the poset $Q_{rn}$ below the condition $q^4$. Recall from Example~\ref{example3} that each point $x_i$ is generic over $V[K]$ for the poset $Q$ below the condition $q$ and the quadruple $\langle V[G_i]\colon i\in 4\rangle$ has dimension characteristic $\langle \infty, n, 1\rangle$. Let $H_i$ for $i\in 4$ be filters on the poset $R$ mutually generic over the model $V[x_i\colon i\in 4]$. By Proposition~\ref{productproposition}, the tuple $\langle V[x_i][H_i]\colon i\in 4\rangle$ has dimension characteristic $\langle \infty, n, 0\rangle$. Write $p_i=\gs/x_i, H_i\in P$. By the forcing theorem applied in the model $V[x_i, H_i]$, $p_i\leq p$ is a condition forcing $x_i\in\tau$.
By the balance assumption on the virtual condition $\bar p$, the set $\{p_i\colon i\in 4\}$ has a common lower bound in the poset $P$. That common lower bound forces each point in the rectangle $\{x_i\colon i\in 4\}$ to belong to the set $\tau$ as required.
\end{proof}

\section{A coloring poset}

The whole development in previous sections would be worthless if no useful posets of interesting dimension characteristics existed. One such a coloring poset is produced in this section.

Fix a number $n\geq 2$ and let $\Gamma_n$ denote the hypergraph on $\mathbb{R}^n$ of arity four consisting of all non-degenerate rectangles. In this section, I define a balanced Suslin poset which adds a total $\Gamma_n$-coloring. After that, I show that this poset does not add a total $\Gamma_{n+1}$-coloring if the ground model is taken to be the choiceless Solovay model. To start, for every subfield $F\subset\mathbb{R}$, say that an algebraic set is visible from $F$ if there is a polynomial defining it whose coefficients all belong to $F$. Define an equivalence relation $E_F$ on $\mathbb{R}^n\setminus F^n$ by connecting points $x_0, x_1$ if either $x_0=x_1$ or there are algebraic sets $A_0, A_1\subseteq\mathbb{R}^n$ visible from $F$ which are distinct from $\mathbb{R}^n$ and $x_0\in A_0$ and $x_1\in A_1$, and there are polynomial functions $f_0, f_1\colon \mathbb{R}^n\to\mathbb{R}^n$ visible from $F$ such that $f_0(x_0)=x_1$ and $f_1(x_1)=x_0$. The connection between this equivalence relation and rectangles is encapsulated in the following simple proposition.

\begin{proposition}
\label{equiproposition}
Let $F\subset\mathbb{R}$ be a subfield and $\{x_i\colon i\in 4\}\subset\mathbb{R}^n$ be a rectangle such that $x_0, x_1\in F^n$. Then

\begin{enumerate}
\item either both points $x_2, x_3$ belong to $F^n$ or neither of them does;
\item if the points $x_2, x_3$ do not belong to $F^n$, then they are $E_F$-equivalent.
\end{enumerate}
\end{proposition}

\begin{proof}
The first item is an immediate corollary of the fact that the fourth point in a rectangle is a linear combination of the other three. For the second item, write $S$ for the segment connecting $x_0$ and $x_1$ and divide into two cases. In the first case, the points $x_0, x_1$ are opposite on the rectangle $R$. In this case, write $y\in F^n$ for the midpoint of $S$, and observe that both $x_2, x_3$ lie on the hypersphere centered at $y$ containing $x_0$. Also, $x_2, x_3$ can be obtained from each other by the reflection about the point $y$. As both the hypersphere and the reflection are visible from $F$, $x_2\mathrel{E_F}x_3$ follows. In the second case, the points $x_0, x_1$ are adjacent on the rectangle $R$. Here, note that $x_2, x_3$ belong to the hyperplanes perpendicular to $S$ and containing $x_0$ or $x_1$ respectively. The points $x_2, x_3$ can be obtained from each other by the reflection about the hypeplane perpendicularly bisecting the segment $S$. As both the hyperplanes and the reflection are visible from $F$, $x_2\mathrel{E_F}x_3$ holds and the proposition follows.
\end{proof}

\noindent The appearance of a hypersphere in this proof is the only reason why Noetherian topologies of this paper are in the algebraic category as opposed to the much simpler affine category. Now, to define the coloring poset, fix a Borel ideal $I$ on $\gw$ containing all singletons and such that it is not generated by countably many sets. A particular choice of $I$ seems to be irrelevant beyond these demands, the summable ideal will do.

\begin{definition}
\label{pndefinition}
The poset $P_n$ consists of all countable functions $p$ such that there is a countable real closed subfield $\supp(p)\subset\mathbb{R}$ such that $\dom(p)=\supp(p)^n$ and $p$ is a $\Gamma_n$-coloring. The ordering is defined by $q\leq p$ if 
$p\subseteq q$ and for every $E_{\supp(p)}$-class $C\subset\dom(q)$, $q\restriction C$ is injective and its range belongs to $I$.
\end{definition}

\begin{proposition}
$P$ is a $\gs$-closed partial ordering.
\end{proposition}

\begin{proof}
Observe that if $q\leq p$ holds then $\dom(q)$ is invariant under the equivalence relation $E_{\supp(p)}$ and on the set $\mathbb{R}^n\setminus\dom(q)$, $E_{\supp(p)}\subseteq E_{\supp(q)}$ holds. The transitivity of the relation $\leq $ on $P$ follows. For the $\gs$-closure, if $\langle p_n\colon n\in\gw\rangle$ is a descending chain of conditions, then $\bigcup_np_n$ is their common lower bound.
\end{proof}

\noindent It is important to find a precise criterion for compatibility of conditions in the poset $P_n$.

\begin{proposition}
\label{compproposition}
Let $p_0, p_1\in P_n$ be conditions. The following are equivalent:

\begin{enumerate}
\item $p_0, p_1$ are compatible;
\item for every point $x_0\in\mathbb{R}^n$ there is a common lower bound of $p_0, p_1$ containing $x_0$ in its domain;
\item 

\begin{enumerate}
\item $p_0\cup p_1$ is a function and $\Gamma_n$-coloring;
\item for every $E_{\supp(p_0)}$-class $C$, $p_1\restriction C$ is injective and its range belongs to $I$;
\item same as (b) except with $p_0$ and $p_1$ interchanged.
\end{enumerate}
\end{enumerate}
\end{proposition}

\begin{proof}
(2) implies (1), which in turn implies (3) by the definition of the ordering on $P_n$. I must argue that (3) implies (2). Let $x_0\in \mathbb{R}^n$ be an arbitrary point.   Choose a countable field $F\subset\mathbb{R}$ containing both $\supp(p_0)$ and $\supp(p_1)$ as subsets and all corrdinates of the point $x_0$ as elements. Write $c=F^n\setminus\dom(p_0\cup p_1)$, $E_0=E_{\supp(p_0)}$, and $E_1=E_{\supp(p_1)}$. Choose an infinite set $b\subset\gw$ in the ideal $I$ which cannot be covered by finitely many singletons and sets of the form $\rng(p_1\restriction C)$ where $C$ is an $E_0$-class, or of the form $\rng(p_0\restriction C)$ where $C$ is an $E_1$-class. This is possible by items (2) and (3) and the choice of the ideal $I$. Now, let $q\colon F^n\to\gw$ be a function extending $p_0\cup p_1$, such that $q\restriction c$ is an injection, and such that for every point $x\in c$, $q(x)\in b\setminus (p_0''[x]_{E_1}\cup p_1''[x]_{E_0})$. This is possible by the claim and the choice of the set $b$. I claim that $q$ is the requested common lower bound of $p_0$ and $p_1$.

First of all, check that $q$ is a $\Gamma$-coloring. Suppose that $R$ is a rectangle in $F^n=\dom(q)$ and argue that it is not $q$-monochromatic. If it has all four vertices in $\dom(p_0\cup p_1)$ then it is not monochromatic by item (a). If it has more than one vertex in the set $c$, then it is not monochromatic as $q\restriction c$ is an injection. If it has exactly one vertex in $c$, call it $x$, then the three others cannot belong to the same condition by the closure properties of $\dom(p_0)$ and $\dom(p_1)$. Thus, there must be a condition (say $p_0$) such that $R$ has two vertices (say $z, w$) in its domain and the third (say $y$) is not in its domain (belongs to $\dom(p_1\setminus p_0)$). I will show that $x\mathrel{E_1}y$ holds; then $q(x)\neq q(y)$ by the choice of the coloring $q$ and $R$ is not monochromatic. There are two cases. Suppose first that the two points $z, w$ are opposite in the rectangle $\{z, w, y, x\}$. Then both points $y, x$ lie on the hypersphere of which the segment between $z, w$ is a diameter, and they can be moved one to another by the reflection around the midpoint of this segment, showing that $x\mathrel{E_1} y$. Suppose now  that the points $z, w$ are adjacent in the rectangle $\{z, w, y, x\}$. Then the points $y, x$ lie on the hyperplanes perpendicular to the segment connecting $z, w$ and containing $z$ and $w$ respectively, and they can be moved to one another by the reflection around the hyperplane perpendicularly bisecting the segment connecting $z$ and $w$. Thus,  $x\mathrel{E_1}y$ holds in this case too.

Finally, prove that $q\leq p_0$ holds; the case $q\leq p_1$ is symmetric. Let $C\subset\dom(q)$ be an $E_0$-class. The choice of the coloring $q$ implies that $q''C\subset b\cup p_1''C$ which then is a set in the ideal $I$. Moreover, $q\restriction C\cap\dom(p_1)$ is an injection by item (b), and $q\restriction C\cap c$ is an injection as well which shares no values with $q\restriction C\cap\dom(p_1)$ by the choice of $q$. In total, $q\restriction C$ is an injection as required.
\end{proof}

\begin{corollary}
$P_n$ is a Suslin forcing.
\end{corollary}

\begin{proof}
It is immediate that the poset $P$ and its order relation are Borel. Proposition~\ref{compproposition} shows that the compatibility relation is Borel as well.
\end{proof}

\begin{corollary}
\label{unioncorollary}
$P_n$ forces the union of the generic filter to be a total $\Gamma_n$-coloring.
\end{corollary}

\begin{proof}
$P_n$ is a $\gs$-closed partial order of partial $\Gamma_n$-colorings. In addition, by Proposition~\ref{compproposition}, for each point $x_0\in \mathbb{R}^n$, the set of conditions containing $x_0$ in its domain is an open dense subset of $P_n$. The corollary follows by a genericity argument.
\end{proof}

\begin{proposition}
\label{balancedproposition}
In the poset $P_n$,
\begin{enumerate}
\item for every total $\Gamma_n$-coloring $c$, the pair $\langle\coll(\gw, \mathbb{R}^n), \check c\rangle$ is balanced;
\item every balanced pair is equivalent to one as in (1);
\item distinct total $\Gamma_n$-colorings yield inequivalent balanced pairs.
\end{enumerate}

\noindent In particular, the poset $P_n$ is balanced if and only if the Continuum Hypothesis holds.
\end{proposition}

\begin{proof}
For item (1), let $c\colon \mathbb{R}^n\to\gw$ be a total $\Gamma_n$-coloring. Let $V[G_0], V[G_1]$ be mutually generic extensions and $p_0, p_1\leq c$ be conditions in $P_n$ in the respective models $V[G_0]$ and $V[G_1]$; I must show that these conditions have a common lower bound. Write $E=E_{\mathbb{R}\cap V}$, $E_0=E_{\mathbb{R}\cap V[G_0]}$, and $E_1=E_{\mathbb{R}\cap V[G_1]}$. The following claim is key.

\begin{claim}
\label{helpclaim}
$E_0\restriction V[G_1]= E\restriction V[G_1]$ and vice versa, $E_1\restriction V[G_0]= E\restriction V[G_0]$.
\end{claim}

\begin{proof}
I prove the first assertion, as the proof of the second is symmetric. The right-to-left inclusion follows from the fact that $V\subset V[G_0]$ holds. For the key left-to-right inclusion, suppose that $x, y\in V[G_1]\setminus V$ are $E_0$-related. This means that there are algebraic sets $A, B\subset\mathbb{R}^n$ visible from $V[G_0]$, distinct from $\mathbb{R}$, containing $x$ and $y$ respectively, and polynomial functions $f, g\colon \mathbb{R}^n\to\mathbb{R}^n$  visible from $V[G_0]$ such that $f(x)=y$ and $g(y)=x$. Now, by Proposition~\ref{productproposition}, $\dim(V[G_0]/V[G_1])\geq n$ holds. By Proposition~\ref{iproposition} there are algebraic sets $A'\subseteq A$, $B'\subseteq B$ visible from $V[G_0]$ and containing $x, y$ respectively, and polynomial functions $f', g'\colon\mathbb{R}^n\to\mathbb{R}^n$ visible in $V$ such that $f\restriction A'=f'\restriction A'$ and $g\restriction B'= g'\restriction B'$. These objects show that $x, y$ are $E$-related as desired.
\end{proof}

\noindent Now it is time to verify item (3) of Proposition~\ref{compproposition} for $p_0, p_1$. First of all, $p_0\cup p_1$ is a function since $p_0\restriction V=p_1\restriction V=c$ and $\dom(p_0)\cap\dom(p_1)\subset V$ holds by the product forcing theorem. To show that $p_0\cup p_1$ is a $\Gamma_n$-coloring, assume that $R$ is a non-degenerate rectangle in its domain and work to show that it is not monochromatic. There are two possible configurations. If one of the conditions $p_0, p_1$ (say $p_0$) contains three points of $R$, then by the closure properties of $\dom(p_0)$ it contains the fourth one as well and $R$ is not monochromatic since $p_0$ is a $\Gamma_n$-coloring. The only remaining option is that $p_0\setminus V$ contains two points (say $x, y$) of the rectangle and $p_1\setminus V$ contains the remaining two (say $z, w$). By Proposition~\ref{equiproposition}, $z, w$ are $E_0$-related, and by Claim~\ref{helpclaim}, they are $E$-related. Then, $p_1(z)\neq p_1(w)$ follows from $p_1\leq c$ and $R$ is not monochromatic. The definition of the ordering $\leq$ in Definition~\ref{pndefinition} is motivated exactly by this step.

Finally, we must show that $p_0''C$ is an injection with range in $I$ for any $E_0$-class $C$ and vice versa, $p_1''D$ is an injection with range in $I$ for any $E_1$-class $D$. The classes $C, D$ are in fact $E$-classes by Claim~\ref{helpclaim}, so this follows from $p_0, p_1\leq c$. Item (3) of Proposition~\ref{compproposition} has been verified, which implies the compatibilty of the conditions $p_0, p_1$ and item (1) of the current proposition.

For item (2), suppose that $\langle Q, \gs\rangle$ is a balanced pair. Strengthening the condition $\gs$ and the poset $Q$ if necessary, I may assume that $Q\Vdash \mathbb{R}\cap V\subset\dom(\gs)$. A balance argument then shows that for each point $x\in\mathbb{R}^n$ there is a specific number $c(x)\in\gw$ such that $Q\Vdash\gs(\check x)=c(x)$. It is immediate that $c\colon\mathbb{R}^n\to\gw$ is a $\Gamma_n$-coloring. I will show that $Q\Vdash\gd\leq\check c$; this will prove that the balanced pairs $\langle Q, \tau\rangle$ and $\langle \coll(\gw, \mathbb{R}^n), \check c\rangle$ are equivalent by \cite[Proposition 5.2.6]{z:geometric} as required. Suppose towards a contradiction that there is a condition $q\in Q$ forcing $\gs\not\leq\check c$. Let $G_0, G_1\subset Q$ be mutually generic filters containing the condition $q$, and write $p_0=\gs/G_0$ and $p_1=\gs/G_1$. Since $p_0\not\leq \check q$, the definition of the ordering in Definition~\ref{pndefinition} shows that there must be a $E_{\mathbb{R}}$-class $C\subset\mathbb{R}^n$ such that $p_0\restriction C$ is not an injection with range in the ideal $I$. Since $\mathbb{R}\cap V\subset\supp(p_1)$ holds, $E_{\mathbb{R}\cap V}\subset E_{\supp(p_1)}$ holds as well and $C$ is a subset of a single $E_{\dom(p_1)}$-class. This shows that the conditions $p_0, p_1$ are incompatible, contradicting the initial assumptions on $Q, \gs$.

Item (3) is immediate. For the last sentence, if the Continuum Hypothesis fails then there are no total $\Gamma_n$-colorings by the result of \cite{komjath:three}, so no balanced virtual conditions by (2) of the present proposition. On the other hand, assume that the Continuum Hypothesis holds and $p\in P_n$ is a condition; I must produce a total $\Gamma_n$-coloring $c$ such that $c\leq p$.
To this end, let $\langle x_\ga\colon \ga\in\gw_1\rangle$ be an enumeration of $\mathbb{R}^n$. By recursion on $\ga$ construct a descending sequence $\langle p_\ga\colon\ga\in\gw_1\rangle$ of conditions such that $p_0=p$, $x_\ga\in \dom(p_{\ga+1})$, and $p_\ga=\bigcup_{\gb\in \ga}p_\gb$. This is straightforward; at the successor stage use Proposition~\ref{compproposition} with the two conditions equal to $p_\ga$. In the end, the coloring $c=\bigcup_\ga p_\ga$ completes the proof.
\end{proof}

\noindent Finally, I must provide the dimension characteristics of the virtual balanced conditions in the poset $P_n$.

\begin{proposition}
\label{dimbalancedproposition}
In the poset $P_n$, every balanced virtual condition is in fact $\langle n+1, n+1, 1\rangle$-dimensionally balanced.
\end{proposition}

\begin{proof}
Suppose that $c\colon\mathbb{R}^n\to\gw$ is a total coloring. Suppose that $\{V[G_i]\colon i\in 4\}$ is a finite collection of generic extensions such that for any three distinct indices $i, j, k\in 4$ it is the case that $\dim(V[G_i, G_j]/V[G_k])\geq n+1$ holds. Suppose that $\{p_i\colon i\in 4\}$ is a collections of conditions in $P_n$ such that $p_i\in V[G_i]$ and $p_i\leq c$ holds;
I must produce a common lower bound of the conditions $p_i$ for $i\in 4$. 

Note that the dimension assumption implies that for pairwise distinct indices $i, j, k\in 4$, $\mathbb{R}\cap V[G_i, G_j]\cap V[G_k]=\mathbb{R}\cap V$ holds by Proposition~\ref{twoproposition}. To set up useful notation, suppose that $i, j\in 4$ are possibly equal indices. Write $F_{ij}\subset\mathbb{R}$ for the smallest subfield of $\mathbb{R}$ containing $\supp(p_i)$ and $\supp(p_j)$ and write $E_{ij}=E_{F_{ij}}$. Also, write $E=E_{\mathbb{R}\cap V}$. The following claim is proved exactly like Claim~\ref{helpclaim} using the assumption on the dimension of the quadruple $\langle V[G_i]\colon i\in 4\rangle$.

\begin{claim}
\label{rclaim}
Let $i, j\in 4$ and let $k$ be distinct from both $i$ and $j$. Then ${E_{ij}}\restriction\dom(p_k)\setminus V=E\restriction\dom(p_k)\setminus V$.
\end{claim}

\noindent To find the requested lower bound, work in the model $V[G_i\colon i\in 4]$ and find a countable subfield $F\subset\mathbb{R}$ containing $\supp(p_i)$ for every $i\in 4$. Write $d=F^n\setminus \dom(\bigcup_ip_i)$. Let $b\subset\gw$ be a set in the ideal $I$ which cannot be covered by a union of finitely many singletons and sets of the form $p_i''C$ where $i\in 4$ and $C\subset\dom(p_i)\setminus V$ is an $E$-class. Note that all of the latter sets belong to the ideal $I$ as $p_i\leq c$ is assumed. Let $q\colon F^n\to\gw$ be any function extending $\bigcup_ip_i$ such that $q\restriction d$ is an injection, and for each point $x\in d$, $q(x)\in b$ and $q(x)$ does not belong to any set of the form $p''_kC$ where $k\in 4$ and $C$ is an $E_{ij}$-class of $x$ for some $i, j\in 4$ distinct from $k$. This is possible by Claim~\ref{rclaim} and the choice of the set $b$. I will show that $q$ is the requested common lower bound of the conditions $p_i$ for $i\in 4$. 

The key point is showing that $q$ is in fact a $\Gamma_n$-coloring. Let $R$ be a rectangle in $\dom(q)$ and work to show that it is not $q$-monochromatic. There are several configurations to investigate.

\noindent\textbf{Case 1.} Suppose that $R\subset\bigcup_{i\in 4}\dom(p_i)$. Let $a\subset 4$ be an inclusion-minimal set such that $R\subset \bigcup_{i\in a}\dom(p_i)$. Divide into subcases:

\noindent\textbf{Case 1.1.} $a$ is a singleton, containing some $i\in 4$. In this case, $R$ is not monochromatic since $p_i$ is a $\Gamma_n$-coloring.

\noindent\textbf{Case 1.2.} $a$ contains exactly two elements $i, j\in 4$. By the closure properties of $\dom(p_i)$ and $\dom(p_j)$, it is impossible for one of them to contain three vertices of $R$, since then it would contain the fourth one also, and Case 1.1 would be in effect. Thus, $\dom(p_i)$ contains vertices $x, y$ and $\dom(p_j)$ contains vertices $z, w$, and none of these vertices can be in $V$.  By Proposition~\ref{equiproposition}, the points $z, w$ are $E_{i}$-equivalent. By Claim~\ref{rclaim}, the points are in fact $E$-equivalent. Finally, since $p_j\leq c$, the points $z, w$ receive distinct $p_j$-color, so $R$ is not monochromatic.

\noindent\textbf{Case 1.3.} $a$ contains exactly three elements $i, j, k\in 4$. Then one of the conditions (say $p_i$) contains two vertices $x, y\in R$, and $\dom(p_j)\setminus V$ and $\dom(p_k)\setminus V$ contain a single element $z$ and $w\in R$ respectively. By the dimension assumption $V[G_i, G_j]\cap V[G_k]=V$ must hold. At the same time, the point $w\in \dom(p_k)\setminus V$ is a linear combination of $x, y, z$, so it should belong to the model $V[G_i, G_j]$, contradicting the case assumption. This case cannot occur no matter whether $R$ is monochromatic or not.

\noindent\textbf{Case 1.4.} $a=4$. Let $x_i\in\dom(p_i)$ for $i\in 4$ be the corresponding points in the rectangle $R$; note that it must be the case that $x_i\in\dom(p_i)\setminus V$. Without loss assume that $x_0$ and $x_1$ are  points adjacent in $R$. Note that the point $x_2$ belongs to the hyperplane $H$ perpendicular to the segment connecting $x_0$ and $x_1$ which contains either $x_0$ or $x_1$. This hyperplane belongs to the model $V[G_0, G_1]$. Since $\dim(V[G_i, G_j]/V[G_k])\geq n+1$, by Proposition~\ref{iproposition} there must be an algebraic set $K\subset H$ in the ground model containing the point $x_2$. Then, the point $x_2$ can be reconstructed in $V[G_3]$ as the closest point on $K$ to $x_3$, since it is even the closest point on $H$ to $x_3$. This means that $x_2\in V[G_3]$, contradicting the case assumption. Thus, this case cannot occur no matter whether $R$ is monochromatic or not.

\noindent\textbf{Case 2.} Suppose that $R$ has exactly one vertex (call it $x$) in the set $d$. Let $a\subset 4$ be an inclusion-minimal set such that the remaining vertices of $R$ are contained in $\bigcup_{i\in a}\dom(p_i)$.

\noindent\textbf{Case 2.1.} $a$ is a singleton containing some index $i$. By the closure properties of $\dom(p_i)$, with the three vertices of $R$ $\dom(p_i)$ would have to contain even $x$, and Case 1.1 would be in force.

\noindent\textbf{Case 2.2.} $a$ contains exactly two elements $i, j\in 4$. Then one of the conditions (say $p_i$) contains two elements of $R$, while $\dom(p_j)\setminus V$ contains a single vertex $y$ of $R$. Note that $y$ is $E_{i}$-equivalent to $x$ by Proposition~\ref{equiproposition}. By the choice of the function $q$, $q(x)\neq q(y)$ holds and $R$ is not monochromatic.

\noindent\textbf{Case 2.3.} $a$ contains exactly three elements $i, j, k\in 4$. Then each of $\dom(p_i)\setminus V$, $\dom(p_j)\setminus V$, and $\dom(p_k)\setminus V$ must contain exactly one vertex of $R$. Let $y$ be the unique vertex of $R$ in $\dom(p_k)\setminus V$. Then $x, y$ are $E_{ij}$-equivalent by Proposition~\ref{equiproposition}, $q(x)\neq q(y)$ holds and $R$ is not monochromatic.

\noindent\textbf{Case 3.} Suppose finally that $R$ has more than one vertex in the set $d$. Then $R$ is not $q$-monochromatic since $q\restriction d$ is an injection.

Finally, I need to show that for every $i\in 4$, $q\leq p_i$ holds. It is clear that $p_i\subset q$ holds. Now, fix an $E_{i}$-class $C\subset\dom(q\setminus p_i)$ and work to show that $q\restriction C$ is an injection with range in the ideal $I$. 

First observe that there is at most one index $j\in a$ such that $C\cap\dom(p_j)\neq 0$. If there were two such indices $j, k$, with points $x_j\in\dom(p_j)\cap C$ and $x_k\in\dom(p_k)\cap C$, then use the definition of the equivalence $E_i$ to show that $x_k\in V[G_i, G_j]$, which is impossible as $V[G_k]\cap V[G_i, G_j]=V$. Thus, let $j$ be the unique index (if it exists) such that $C\cap \dom(p_j)\neq 0$. By Claim~\ref{rclaim}, $C\cap \dom(p_j)$ is a single $E$-class disjoint from $V$, so $p_j\restriction (C\cap\dom(p_j))$ is an injection with range in $I$. Also, $q\restriction C\setminus\dom(p_j)$ is an injection with range in $I$, as $C\setminus\dom(q_j)\subset d$ and $q\restriction d$ is an injection with range included in the set $b\in I$. Finally, for every point $x\in C\setminus\dom(p_j)$ the color $q(x)$ was chosen exactly so that it is different from all colors in the set $p_j''(C\cap \dom(p_j))$. Thus $q\restriction C$ is an injection with range in the ideal $I$ as desired.
\end{proof}

\noindent Finally, I am in a position to prove Theorem~\ref{maintheorem}. Let $n\geq 2$ be a number. Let $\kappa$ be an inaccessible cardinal, and let $W$ be the choiceless Solovay model derived from $\kappa$. Let $G\subset P_n$ be a filter generic over $W$, and consider the model $W[G]$. The poset $P_n$ is $\gs$-closed, therefore the model $W[G]$, as a $\gs$-closed extension of a model of DC, satisfies DC as well. By Corollary~\ref{unioncorollary}, in $W[G]$ the chromatic number of $\Gamma_n$ is countable. By the conjunction of Propositions~\ref{balancedproposition} and~\ref{dimbalancedproposition} and Theorem~\ref{preservationtheorem}, in the model $W[G]$, every non-meager subset of $\mathbb{R}^{n+1}$ contains vertices of a non-degenerate rectangle, theorefore the chromatic number of $\Gamma_{n+1}$ is uncountable. Theorem~\ref{maintheorem} has just been proved.

\bibliographystyle{plain} 
\bibliography{odkazy,zapletal,shelah}

\end{document}